\documentclass{amsart}
\usepackage{mathrsfs}
\usepackage{stmaryrd,mathtools}
\usepackage{enumerate}
\usepackage{tikz-cd}
\usepackage[all]{xy}
\usepackage{aliascnt}

\usepackage{shuffle}
\usepackage{ytableau}

\usepackage[colorlinks=true, linkcolor=blue, citecolor=blue]{hyperref}
\usepackage{fullpage}
\usepackage{verbatim}
\usepackage{amssymb}

\newtheorem{theorem}{Theorem}[section]

\newaliascnt{lemma}{theorem}
\newtheorem{lemma}[lemma]{Lemma}
\aliascntresetthe{lemma}

\newaliascnt{corollary}{theorem}
\newtheorem{corollary}[corollary]{Corollary}
\aliascntresetthe{corollary}

\newaliascnt{proposition}{theorem}
\newtheorem{proposition}[proposition]{Proposition}
\aliascntresetthe{proposition}

\newaliascnt{potato}{theorem}

\aliascntresetthe{potato}

\newaliascnt{definitionlemma}{theorem}

\aliascntresetthe{definitionlemma}

\newaliascnt{conjecture}{theorem}

\aliascntresetthe{conjecture}

\newaliascnt{question}{theorem}

\aliascntresetthe{question}

\theoremstyle{definition}

\newaliascnt{definition}{theorem}
\newtheorem{definition}[definition]{Definition}
\aliascntresetthe{definition}

\newaliascnt{remark}{theorem}
\newtheorem{remark}[remark]{Remark}
\aliascntresetthe{remark}

\newaliascnt{example}{theorem}

\aliascntresetthe{example}

\newaliascnt{notation}{theorem}

\aliascntresetthe{notation}

\newtheorem*{acknowledgements}{Acknowledgements}

\definecolor{darkblue}{rgb}{0.6,0,0.1}

\usepackage{tikz}
\usetikzlibrary{calc, shapes, backgrounds,arrows,positioning,plotmarks}
\tikzset{>=stealth',
	head/.style = {fill = white, text=black},
	plaque/.style = {draw, rectangle, minimum size = 10mm}, 
	pil/.style={->,thick},
	junct/.style = {draw,circle,inner sep=0.5pt,outer sep=0pt, fill=black}
}

\newcommand{\bG}{\mathbb{G}}
\newcommand{\bP}{\mathbb{P}}

\newcommand{\bZ}{\mathbb{Z}}

\newcommand{\cZ}{\mathcal{Z}}
\newcommand{\cX}{\mathcal{X}}

\newcommand{\cY}{\mathcal{Y}}

\newcommand{\cO}{\mathcal{O}}

\newcommand{\cE}{\mathcal{E}}
\newcommand{\cH}{\mathcal{H}}

\newcommand{\cL}{\mathcal{L}}
\newcommand{\cM}{\mathcal{M}}

\newcommand{\cQ}{\mathcal{Q}}

\newcommand{\cN}{\mathcal{N}}

\newcommand{\sI}{\mathscr{I}}

\newcommand{\red}{\mathrm{red}}
\newcommand{\sm}{\mathrm{sm}}

\DeclareMathOperator{\Gr}{Gr}

\DeclareMathOperator{\ps}{ps}
\DeclareMathOperator{\s}{s}
\DeclareMathOperator{\Isom}{Isom}
\DeclareMathOperator{\uIsom}{\underline{\Isom}}

\DeclareMathOperator{\GL}{GL}

\DeclareMathOperator{\Band}{Band}
\DeclareMathOperator{\Spec}{Spec}

\DeclareMathOperator{\Var}{Var}

\newcommand{\thickslash}{\mathbin{\!\!\pmb{\fatslash}}}

\makeatletter
\@namedef{subjclassname@2020}{%
	\textup{2020} Mathematics Subject Classification}
\makeatother

\newif\ifhascomments \hascommentstrue
\ifhascomments
\newcommand{\matt}[1]{{\color{red}[[\ensuremath{\spadesuit\spadesuit\spadesuit} #1]]}}
\newcommand{\dori}[1]{{\color{blue}[[\ensuremath{\clubsuit\clubsuit\clubsuit} #1]]}}
\newcommand{\elden}[1]{{\color{blue}[[\ensuremath{\clubsuit\clubsuit\clubsuit} #1]]}}
\else
\newcommand{\elden}[1]{}
\newcommand{\dori}[1]{}
\newcommand{\matt}[1]{}
\fi

\begin{document}
	
	\title{Proper splittings and projectivity for good moduli spaces}
	
	\author{Dori Bejleri}
	\thanks{DB was partially supported by NSF grant DMS-2401483.}
	\address[DB]{Department of Mathematics, University of Maryland, College Park MD 20742, USA}
	\email{dbejleri@umd.edu}
	
	\author{Elden Elmanto}
	\thanks{}
	\address[EE]{Department of Mathematics, University of Toronto, Toronto ON M5S2E4, Canada}
	\email{elden.elmanto@utoronto.ca}
	
	\author{Matthew~Satriano}
	\thanks{MS was partially supported by a Discovery Grant from the
		National Science and Engineering Research Council of Canada and a Mathematics Faculty Research Chair from the University of Waterloo.}
	\address[MS]{Department of Pure Mathematics, University
		of Waterloo, Waterloo ON N2L3G1, Canada}
	\email{msatrian@uwaterloo.ca}
	
	\date{\today}
	\subjclass[2020]{}

	\begin{abstract} We show that any good moduli space $\pi \colon \cX \to Y$ has a splitting after a proper, generically finite covering of $Y$. As an application we generalize Koll\'ar's ampleness lemma to give a criterion for projectivity of a good moduli space. 
	\end{abstract}
	
	\maketitle

	\numberwithin{theorem}{section}
	\numberwithin{lemma}{section}
	\numberwithin{corollary}{section}
	\numberwithin{proposition}{section}
	\numberwithin{conjecture}{section}
	\numberwithin{question}{section}
	\numberwithin{remark}{section}
	\numberwithin{definition}{section}
	\numberwithin{example}{section}
	\numberwithin{notation}{section}
	\numberwithin{equation}{section}
	
	\section{Introduction}
	\label{sec:intro}
	
	An essential tool in the study of Deligne-Mumford stacks is the existence of finite covers by schemes (e.g. \cite[Theorem 16.6]{LMB}, \cite[Theorem 2.7]{EHKV}, \cite{Olsson2005}). By contrast, no such covers exist for general Artin stacks. The goal of this paper is to offer a replacement for stacks with a good moduli space.  More precisely, if $\pi \colon \cX \to Y$ is a good moduli space with $Y$ separated, we show that there exists a scheme $Y'$ and a morphism $Y' \to \cX$ such that $Y' \to Y$ is proper, surjective and generically finite (Theorem \ref{thm:main-gms-splitting}). While the map $Y' \to \cX$ is not surjective, it is surjective on $\textbf{S}$-equivalence classes. 
	
	Our main application is a generalization of Koll\'ar's ampleness lemma to good moduli spaces (Theorem \ref{thm:kollar-ampleness}). The modern GIT-free approach to constructing moduli spaces proceeds in three steps: construct an algebraic stack $\cX$ via Artin's axioms, show that the stack admits a moduli space $Y$ \cite{keelmori, AHLH}, produce an ample line bundle to show $Y$ is projective. Theorem \ref{thm:kollar-ampleness} gives a general criterion to carry out the third step. Theorems \ref{thm:main-gms-splitting} and \ref{thm:kollar-ampleness} are generalizations of the strategy used to prove projectivity of good moduli spaces of Calabi-Yau  varieties in \cite{ABBDILW} (especially \cite[Theorem 14.11]{ABBDILW}). 
	
	\subsection{Splittings of good moduli spaces by alterations}
	
	Recall that a map between $X \to Y$ between integral algebraic spaces is an \emph{alteration} if it is proper, surjective, and generically finite. More generally, a map of Noetherian algebraic spaces $X \to Y$ is an \emph{alteration} if it is proper, surjective, generically finite and every irreducible component of $X$ dominates an irreducible component of $Y$.

	\begin{theorem}\label{thm:main-gms-splitting}
		Let $\pi\colon\cX\to Y$ be a good moduli space where 
		$\cX$ is an Artin stack with affine stabilizers and separated diagonal over  $Y$, and of finite type over a Noetherian scheme $S$. If $f\colon Z\to Y$ is a morphism of finite type from a separated algebraic space over $S$, then there exists an alteration $g\colon Z'\to Z$ and a commutative diagram
		\[
		\xymatrix{
			Z'\ar[r]\ar[d]_-{g} & \cX\ar[d]^-{\pi}\\
			Z\ar[r]^-{f} & Y
		}
		\]
	\end{theorem}
	
	\begin{remark}\label{rmk:properlystable-factorization}
		\begin{enumerate}[(i)]
			\item The proof of Theorem \ref{thm:main-gms-splitting} shows that $Z'\to\cX$ factors through the closure of the polystable locus $\cX^{\ps}$ (see Definition \ref{def:ps} and Lemma \ref{lem:ps}).
			\item By Chow's lemma for algebraic spaces, we may take $Z'$ to be quasi-projective over $S$. 
			\item If $S$ is finite type over a quasi-excellent scheme $S_0$ of dimension $\le 3$, then we may take $Z'$ to be regular by \cite[Theorem 1.2.5]{Temkin2017}.
		\end{enumerate}
	\end{remark}
	
	It is possible to loosen the Noetherian hypothesis on $S$. In general, Noetherian approximation does not work well for good moduli spaces and so some assumptions are still required, see \cite[\S7]{AHRetalelocal}. We recall the following conditions on an Artin stack $\cY$ introduced in \cite[\S15]{AHRetalelocal}: (FC) there are only a finite number of different characteristics in $\cY$, (PC) every closed point of $\cY$ has positive characteristic, and (N) every closed point of $\cY$ has nice stabilizer (i.e., there is an open and closed multiplicative type normal subgroup of the stabilizer such that the quotient by this subgroup is finite, locally constant, and has order prime to all characteristics). An application of \cite{Rydh2015, RydhNoetherianAbsolute} yields the following.

	\begin{theorem}\label{thm:main-gms-splitting-nonNoeth}
		Theorem \ref{thm:main-gms-splitting} holds when instead of requiring $S$ to be Noetherian, one assumes $S$ is quasi-compact and quasi-separated, $\cX \to S$ is finitely presented, and one of the following holds:~$S$ is $(\mathrm{FC})$, $\cX$ is $(\mathrm{PC})$, or $\cX$ is $(\mathrm{N})$. In this case, $g$ is proper, generically finite and surjective. 
	\end{theorem}

	Furthermore, we obtain a variant of Theorem \ref{thm:main-gms-splitting} where we require that our proper splitting extends a given rational map $Z \dashrightarrow \cX$. This is the analogue of the fact that a rational map to a proper scheme can be resolved by blowing up some ideal and it implies that families with polystable generic fiber can be compactified after an alteration and that the variation (Definition \ref{def:variation}) of a family parametrized by a good moduli space behaves well. 
	
	\begin{theorem}\label{thm:main-gms-splitting-rational-version}
		Keep the notation and hypotheses of Theorem \ref{thm:main-gms-splitting} and assume in addition that $S$ is of finite type over a quasi-excellent scheme of dimension at most $3$. Additionally suppose $Z$ is irreducible and there is a rational map $\widetilde{f}\colon Z\dasharrow\cX$ such that the generic point of $Z$ maps to the polystable locus $\cX^{\ps}$ and $\pi\widetilde{f}=f$. Then the conclusion of Theorem \ref{thm:main-gms-splitting} holds where we also require that $Z'\to\cX$ generically agrees with $\widetilde{f}g$. 
	\end{theorem}
	
	\begin{corollary}\label{cor:compactification}
		Under the assumptions of Theorem \ref{thm:main-gms-splitting-rational-version}, suppose further that $Y$ is proper. Let $U$ be an irreducible separated algebraic space of finite type over $S$ and suppose $U \to \cX$ is a morphism with generic point mapping to the polystable locus. Then there exists an alteration $U' \to U$, a compactification $U' \subset Z'$ over $S$ and an extension $Z' \to \cX$. 
	\end{corollary}
	
	\begin{remark}
		By \cite[Theorem 1.2.5]{Temkin2017}, we can pick $U' \subset Z'$ to be a normal crossings compactification of a regular scheme. 
	\end{remark}
	
	\begin{remark}
		The assumption on $S$ in Theorem \ref{thm:main-gms-splitting-rational-version} is used only to handle the case of $\bG_m$-gerbes $\cX\to Y$, where our argument makes use of regular alterations obtained in \cite[Theorem 1.2.5]{Temkin2017} and purity for Brauer groups \cite[Theorem 3.5.4]{ColliotTheleneSkorobogatov}. Note it holds in particular when $S$ is the spectrum of a field, a discrete valuation ring, or a number ring. 
	\end{remark}
	
	\begin{definition}\label{def:variation}
		Let $\cX \to Y$ be good moduli space with assumptions of Theorem \ref{thm:main-gms-splitting-rational-version}. If $g \colon Z \to \cX$ a morphism from an irreducible algebraic space of finite type over $S$, we define the variation $\Var g$ to be the dimension of the scheme theoretic image of the composition $f \colon Z \to Y$. 
	\end{definition}
	
	\begin{corollary}\label{cor:variation}
		Keeping the assumptions as in Theorem \ref{thm:main-gms-splitting-rational-version}, suppose $Y$ is separated over $S$ and $g \colon Z \to \cX$ is a morphism from an irreducible separated algebraic space over $S$. Then there exists an alteration $Z' \to Z$, a dominant morphism $Z' \to Z''$ and commutative diagram
		$$
		\xymatrix{Z' \ar[r] \ar[d] & Z'' \ar[d]^{g''} \\ Z \ar[r]_g & \cX}
		$$
		such that $\Var g = \Var g'' = \dim Z''$.
	\end{corollary}
	
	\subsection{Generalization of Koll\'ar's Ampleness Lemma:~projectivity of good moduli spaces}
	
	Our main application generalizes Koll\'ar's celebrated ampleness lemma \cite[Lemma 3.9]{KollarAmpleness} \& \cite[Theorem 4.5]{KovacsPatakfalvi2017} to the good moduli space of an Artin stack. Our result gives a criterion for a good moduli space to be projective. We assume in this subsection that $S = \Spec k$ is a field since ampleness of a line bundle on $Y$ over $S$ can be checked fiberwise over $S$.

	\begin{definition}\label{def:semiposvb}
		A finite rank vector bundle $\cE$ on a stack $\cX$ is \emph{semi-positive} (resp.~\emph{weakly semi-positive}) if for every map $f\colon C\to\cX$ from a smooth curve (resp.~whose generic point maps to the polystable locus $\cX^{\ps}$), and every surjection $f^*\cE\twoheadrightarrow\cL$ with $\cL$ a line bundle, we have $\deg(\cL)\geq0$.
	\end{definition}
	
	\begin{remark}\label{rmk:semiposvbiffnef}
		If $\cX$ is an algebraic space, then $\cE$ is semi-positivity if and only if it is nef, i.e., $\cO(1)$ on $\bP_\cX(\cE)$ is nef. This is because every quotient line bundle of $f^*\cE$ is the pullback of $\cO(1)$. When $\cX$ is a stack, however, there is no natural definition of nefness, and hence we instead consider the semi-positivity condition defined above. We further remark that weak semi-positivity is equivalent to semi-positivity when $\cX^{\ps}=\cX$, e.g., for gerbes over separated stacks.
	\end{remark}

	Let $\cE$ be a vector bundle of rank $n$ on a stack $\cX$ with structure group $\rho\colon G\to\GL_n$, and let $\alpha\colon\cE \twoheadrightarrow \cQ$ be a surjection to a locally free sheaf of rank $q$. Then we have an associated \emph{classifying map}
	\[
	[\alpha]\colon\cX\to[\Gr(n,q)/G]
	\]
	where $\Gr(n,q)$ is the Grassmannian. 
	
	\begin{theorem}\label{thm:kollar-ampleness}
		Let $\pi\colon\cX\to Y$ be a good moduli space with $Y$ proper (equivalently, by \cite[Theorem A]{AHLH}, $\cX$ is $\Theta$-reductive, $\textbf{S}$-complete, and satisfies the existence part of the valuative criterion). Let $\cE$ be a finite rank vector bundle and $\alpha\colon\cE\twoheadrightarrow\cQ$ a surjection to a locally free sheaf. Assume
		\begin{enumerate}
			\item there exists $N>0$ such that $\det(\cQ)^{\otimes N}$ descends to a line bundle $\cL$ on $Y$,
			\item the fibers of $[\alpha]$ contain finitely many closed points, and
			\item one of the following holds:
			\begin{enumerate}
				\item $\cE$ is weakly semi-positive with linearly reductive structure group $G$, or
				\item\label{kollar::Schur-positive} $\cE=\rho(\cE_1,\dots,\cE_k)$ where the $\cE_i$ are weakly semi-positive vector bundles of rank $n_i$, $G=\prod_i\GL_{n_i}$ and $\rho$ is a semi-positive polynomial representation, or
				\item for every map $f\colon C\to\cX$ from a curve whose generic point maps to the polystable locus $\cX^{\ps}$, $f^*\cE\simeq\cE'\otimes\cM$, where $\cM$ has degree $0$ and $\cE'$ is of the form (\ref{kollar::Schur-positive}).
			\end{enumerate}
		\end{enumerate}
		Then $\cL$ is ample, hence $Y$ is projective.
	\end{theorem}

	\acknowledgements It is a pleasure to thank Kenny Ascher, Brian Conrad, Dan Edidin, Jack Hall, Giovanni Inchiostro, Sam Molcho, and Jerry Wang for helpful conversations. We would especially like to thank Sid Mathur for his suggestion to consider the stack $\sI$ in the proof of Proposition \ref{prop:gerbe-abelian}.

	\section{Preliminary reductions for gerbes}
	
	Recall that a group scheme $G\to S$ is \emph{linearly reductive} if $BG \to S$ is a good moduli space. It is \emph{reductive} if it is smooth and affine, and the geometric fibers are connected reductive groups. Note that linearly reductive group schemes need not be smooth nor connected. 
	Consider the functor $Z(G)$ whose $T$-points are given by the center of $G(T)$.
	Theorem 3.3.4 of \cite{ConradReductive} shows that $Z(G)$ is representable by a closed subgroup scheme $Z(G)\subset G$, that $Z(G)\to S$ is flat, the formation of $Z(G)$ commutes with base change on $Y$, and that $Z(G)\to S$ is of multiplicative type; note that $Z(G)\to S$ need not be smooth nor connected.
	
	\begin{proposition}\label{prop:gerbe-abelian}
		Let $Y$ be an algebraic space and $G\to Y$ a reductive group scheme. Suppose that for all gerbes $\mathcal{Z} \to Y$ banded by $Z(G)$, there exists an alteration $Y' \to Y$ and a lift $Y' \to \mathcal{Z}$. Then the same holds for all gerbes $\cX \to Y$ banded by $G$. 
		
	\end{proposition}
	\begin{proof}
		Let $\cX \to Y$ be a gerbe and fix an isomorphism $\varphi \colon \Band(\cX) \to \Band(G)$. We produce a $Z(G)$-gerbe associated to $\cX$ by a generalization of the argument in the proof of \cite[Theorem 58]{sid}. Consider the morphism of stacks over $Y$ 
		\begin{equation}\label{eqn:isom}
			\uIsom_Y(\cX, BG) \to \uIsom_Y(\Band(\cX), \Band(G)).
		\end{equation}
		Pulling back to an fppf cover $U \to Y$ where $\cX|_U = BG_{\cX}$, the above morphism becomes 
		$$
		[\uIsom_{gp/U}(G_{\cX}, G)/G] \to \uIsom_{gp/U}(G_{\cX},G)/G^{ad}
		$$
		where $G^{ad} = G/Z(G)$. This morphism is representable by algebraic stacks by \cite[Cor. XXIV 1.8]{SGA3} and is a gerbe for $Z(G)$. Thus by \cite[Tag 06DB]{stacks-project}, the morphism in (\ref{eqn:isom}) is also representable by algebraic stacks and a $Z(G)$-gerbe. The choice of isomorphism $\varphi$ gives a global section of $\uIsom_Y(\Band(\cX), \Band(G))$ and pulling back along (\ref{eqn:isom}) gives us a $Z(G)$-gerbe $\sI :=\uIsom_Y^{\varphi}(\cX, BG) \to X$ of banded equivalences between $BG$ and $\cX$.
		
		Suppose that there exists an alteration $Y' \to Y$ and a lift $Y' \to \sI$. Then $Y' \to \sI$ gives an isomorphism $BG_{Y'} \to \cX|_{Y'}$. Since the former has a section, so does the latter and thus by composition we have a lift $Y' \to \cX$. 
		
	\end{proof}

	The next result will be useful to reduce the case of gerbes to the special case of gerbes with connected stabilizers. Theorem 9.9 of \cite{AHRetalelocal} proves that if $G\to S$ is a flat separated group algebraic space with affine fibers such that the connected component of the identity of each fiber $(G_s)^0$ is linearly reductive, then there exists a unique linearly reductive characteristic closed subgroup $G^0_{\sm}\hookrightarrow G$ which is smooth over $S$, restricts to $(G_s)^0_{\mathrm{red}}$ on all fibers, and such that $G/G^0_{\sm}$ is quasi-finite and separated over $S$. 
	
	\begin{remark}\label{rmk:G0smoverX}
		We note that the construction of $G^0_{\sm}$ commutes with base change. Indeed, it suffices to check this   when $S=\Spec\kappa$ is a field since $G^0_{\sm}$ is characterized by a fiberwise condition \cite[Theorem 9.9]{AHRetalelocal}. 
		Then we have $G^0_{\sm}=G^0_{\red}$ which is smooth over 
		$\kappa$, hence geometrically reduced and geometrically irreducible (equivalently geometrically connected); thus, $G^0_{\sm}\times_\kappa \kappa'=(G\times_\kappa \kappa')^0_{\sm}$. 
		In particular, by descent, we see that if $G\to\cX$ is a flat separated group algebraic space such that for all field valued points $x\colon\Spec k\to\cX$, the stabilizer $G_x$ is affine and its connected component $(G_x)^0$ is linearly reductive, then there exists a unique linearly reductive characteristic closed subgroup $G^0_{\sm}\hookrightarrow G$ which is smooth over $\cX$, restricts to $(G_x)^0_{\mathrm{red}}$ on all field valued points, and $G/G^0_{\sm}$ is quasi-finite and separated over $\cX$.
	\end{remark}

	\begin{proposition}\label{prop:gerbe-connected}
		Let $\pi\colon\cX\to Y$ be a good moduli space gerbe where $\cX$ is a finitely presented Artin stack with affine diagonal over a scheme $S$. Letting $\overline{\cX}\colon=\cX\thickslash(I\cX)^0_{\sm}$ be the rigidification, $\pi$ factors as $\cX\xrightarrow{\gamma}\overline{\cX}\xrightarrow{\overline{\pi}} Y$, where $\gamma$ is a good moduli gerbe for $(I\cX)^0_{\sm}$ and $\overline{\pi}$ a tame coarse moduli space.
	\end{proposition}
	\begin{proof}
		First, $(I\cX)^0_{\sm}$ exists by Remark \ref{rmk:G0smoverX} and since $I\cX\to\cX$ is an affine linearly reductive group scheme.
		
		By \cite[Tag 06QJ]{stacks-project}, $I\cX\to\cX$ is flat and locally finitely presented, hence finitely presented as $\pi$ is. Since $(I\cX)^0_{\sm}$ is smooth over $\cX$ (hence flat), and a characteristic subgroup (hence normal) in $I\cX$, the rigidification $\overline{\cX}$ exists by \cite[Theorem A.1]{AbramovichOlssonVistoliTame}. 
		Thus, we may factor $\pi$ as $\cX\xrightarrow{\gamma}\overline{\cX}\xrightarrow{\overline{\pi}} Y$, where $\gamma$ and $\overline{\pi}$ are both gerbes. Furthermore, $\gamma$ is a good moduli map because on an fppf cover $V\to\overline{\cX}$ by an algebraic space, we have $\cX\times_{\overline{\cX}}V=B(I(\cX|_V)^0_{\sm})$ and $I(\cX|_V)^0_{\sm}$ is linearly reductive by \cite[Theorem 9.9]{AHRetalelocal}. This implies $\overline{\pi}$ is a good moduli space as both the cohomological affine and Stein properties of $\overline{\pi}$ follow immediately from those of $\gamma$ and $\pi$.
		
		By \cite[Remark A.2]{AbramovichOlssonVistoliTame}, $\overline{G}\colon=I\cX/(I\cX)^0_{\sm}$ descends to the inertia $I\overline{\cX} \to \overline{\cX}$. By Remark \ref{rmk:G0smoverX} as well as \cite[Theorem 9.9]{AHRetalelocal}, $\overline{G}\to\cX$ is quasi-finite and separated, so $\overline{\cX}$ has quasi-finite separated inertia. Since $\pi$ is a good moduli space, $\cX$ has unpunctured inertia \cite[Theorem 4.1]{AHLH}. From the definition of unpunctured inertia, this implies $\overline{\cX}$ does as well. Since $\overline{\cX}$ has quasi-finite unpunctured inertia, \cite[Proposition 3.54]{AHRetalelocal} tells us that $\overline{\cX}$ has finite inertia. Combined with the fact that $\overline{\pi}$ is a good moduli space, we see $\overline{\pi}$ is a tame coarse space space.
	\end{proof}

	\begin{corollary}\label{cor:gerbe-conn-split}
		In the situation of Proposition \ref{prop:gerbe-connected}, if we further assume $S$ is Noetherian, then there exists a finite cover $Z \to Y$ and a lift $Z \to \overline{\cX}$.  
	\end{corollary}
	
	\begin{proof}
		By \cite[Theorem 2.7]{EHKV}, there exists a scheme $Z$ and a finite surjective morphism $Z \to \overline{\cX}$. Since $\overline{\cX} \to Y$ is a coarse moduli space, it is proper and quasi-finite so the composition $Z \to Y$ is finite. 
	\end{proof}
	
	\section{Proof of Theorem \ref{thm:main-gms-splitting}}
	
	Throughout this section, we use the running notation from the statement of Theorem \ref{thm:main-gms-splitting}. 
	By \cite[Theorem 6.1]{AHRetalelocal}, $\pi\colon\cX\to Y$ has affine diagonal and $Y\to S$ is finitely presented.
	
	\begin{lemma}\label{l:SNoeth}
		To prove Theorems \ref{thm:main-gms-splitting} and \ref{thm:main-gms-splitting-nonNoeth}, it suffices to prove Theorem \ref{thm:main-gms-splitting} where $Z\to Y$ is the identity map, $Y$ is irreducible and separated over $S$, and $S$ is finite type over a localization of $\Spec\bZ$.
	\end{lemma}
	\begin{proof}
		
		First, by \cite[Proposition 4.7]{Alper}, the base change $\cX\times_YZ\to Z$ is a good moduli space; it is finitely presented with affine diagonal, so it suffices to assume $Z\to Y$ is the identity map.
		
		Let $S_0$ be the semi-localization of $\Spec\bZ$ in the characteristics appearing in $S$ if $S$ satisfies (FC), or otherwise let $S_0 = \Spec \bZ$. By \cite[Tag 07SU]{stacks-project}, we may write $S=\lim_i S_i$ where each $S_i$ is a quasi-separated and finite type scheme over $S_0$ and all transition maps $S_i\to S_j$ are affine. Then Propositions B.2 and B.3 of \cite{Rydh2015} imply there is a compatible family of finitely presented Artin stacks $\cX_i$, finitely presented algebraic spaces $Y_i$ and maps $\cX_i\xrightarrow{\pi_i} Y_i\to S_i$ where $\pi_i$ has affine diagonal and the base change of $\cX_i\xrightarrow{\pi_i} Y_i\to S_i$ is $\cX\xrightarrow{\pi} Y\to S$. By \cite[Theorem 5.10]{RydhNoetherianAbsolute}, there is an index $i$ for which $\cX_i$ admits a good moduli space $\cX_i\to X_i$. By the universal property of good moduli space (using that $\cX_i$ is locally Noetherian), \cite[Theorem 6.6]{Alper} tells us $\pi_i$ factors as $\cX_i\xrightarrow{\alpha_i} X_i\xrightarrow{\beta_i}Y_i$. Letting $\cX\xrightarrow{\alpha} X\xrightarrow{\beta}Y$ be the base change from $S_i$ to $S$, \cite[Proposition 4.7]{Alper} implies $\alpha$ is a good moduli space, hence $\beta$ is an isomorphism. Thus, \cite[Propositions B.3]{Rydh2015} tells us $\beta_i$ is an isomorphism for sufficiently large $i$.
		
		Thus we may assume $S$ is finite type over $S_0$ and in particular Noetherian, so $Y$ has finitely many irreducible components. Letting $Y'$ be the disjoint union of the connected components, we have an alteration $Y'\to Y$, so it suffices to consider each irreducible component separately.
	\end{proof}
	
	Before the next result, we make some definitions we need to apply the results of \cite{ER}. 
	
	\begin{definition}\label{def:ps}
		Let $\pi \colon \cX \to Y$ be a good moduli space. We say a point $x \in \cX$ is
		\begin{enumerate}[(i)]
			\item \emph{stable} if $\pi^{-1}(\pi(x)) = \{x\}$, and
			\item \emph{polystable} if $\{x\}$ is closed inside $\pi^{-1}(\pi(x))$. 
		\end{enumerate}
	\end{definition}
	
	\begin{lemma}\label{lem:ps} Let $\pi \colon \cX \to Y$ be a good moduli space with $\cX$ locally Noetherian. 
		\begin{enumerate}[(i)]
			\item The set of stable points form an open substack $\cX^{\s} \subset \cX$. 
			\item The set of polystable ponts for a constructible substack $\cX^{\ps} \subset \cX$ containing $\cX^{\s}$. 
			\item For any algebraically closed field $k$, the good moduli map induces a bijection $\cX^{\ps}(k) \to Y(k)$. 
		\end{enumerate}
	\end{lemma}
	\begin{proof} Part (i) is \cite[Proposition 2.6]{ER}.
		
		For part (ii), note that the polystable locus commutes with base change over $Y$ since the formation of good moduli spaces does. Moreover, constructibility can be checked \'etale locally so by \cite[Theorem 6.1]{AHRetalelocal}, we are reduced to the case where $\cX = [\Spec A/\GL_n]$ where one can argue as in the proof of \cite[Theorem 7.8]{xuzhuang}: $\cX^{\ps} = [V/\GL_n]$ where $V$ is the constructible subset
		$$
		V = \{x \in \Spec \mid \dim G \cdot x < \dim G \cdot y \text{ for all } y \in \pi^{-1}(\pi(x))\}. 
		$$
		
		For part (iii), note that $\cX^{ps}(k)$ makes sense assuming part (ii). Then the bijection follows from \cite[Theorem 4.16(iv) \& Proposition 9.1]{Alper}.  
		
	\end{proof}

	\begin{proposition}\label{prop:gerbe-case}
		To prove Theorem \ref{thm:main-gms-splitting}, we may assume $Z\to Y$ is the identity map, $Y$ is irreducible, and $\pi$ is a gerbe.
		
	\end{proposition}
	\begin{proof}
		By Lemma \ref{l:SNoeth}, we may assume $S$ is Noetherian, $Z\to Y$ is the identity map, and $Y$ is irreducible. Our initial goal is to apply \cite[Theorem 2.11]{ER}. To do so we argue as in the proof of \cite[Theorem 14.11]{ABBDILW} to first reduce to a setting where the hypotheses of the theorem hold. Let $\cX^{\ps}\subset\cX$ be the polystable locus, and if $\cZ\subset\cX$ denotes its closure, then the composition $\cZ\to\cX\to Y$ is a good moduli space map by \cite[Lemma 4.14]{Alper} and Lemma \ref{lem:ps}(iii). Thus, it suffices to prove Theorem \ref{thm:main-gms-splitting} under the assumption that $\cX^{\ps}$ is dense. Since $\pi$ is finitely presented, $\cX$ has finitely many irreducible components. Thus, some irreducible component $\cZ\subset\cX$ has $\pi(\cZ)=Y$, and hence, $\cZ\to Y$ is a good moduli space. We may therefore assume $\cX$ is irreducible. We may further replace $\cX$ by its reduction. Let $\eta$ be the generic point of $\cX$. Since $\cX^{\ps}$ is dense and constructible (Lemma \ref{lem:ps}(ii)), it must contain $\eta$ so $\eta$ is polystable. Thus it is both closed and dense in $\pi^{-1}(\pi(\eta))$ so $\eta$ is in fact stable. Thus $\cX^{\s} \subset \cX$ is open and dense (Lemma \ref{lem:ps}(i)). 
		
		We may therefore apply \cite[Theorem 2.11]{ER}, which yields a commutative diagram
		\[
		\xymatrix{
			\cX'\ar[r]^-{p'}\ar[d]_-{\pi'} & \cX\ar[d]^-{\pi}\\
			Y'\ar[r]^p & Y
		}
		\]
		where $p'$ is a sequence of saturated blow-ups, $p$ is a sequence of blow-ups, $\cX'$ has affine diagonal (an implicit assumption in their paper), and $\pi'$ factors as a gerbe over a tame Artin stack followed by a coarse space map. Thus, replacing $Y$ with $Y'$, we may assume $\pi$ itself factors as $\cX\xrightarrow{\alpha}\cY\xrightarrow{\beta} Y$ with $\alpha$ a gerbe, $\beta$ a coarse space map, and $\cY$ a tame Artin stack. Since $\cY\to S$ has affine diagonal (an implicit assumption in \cite{ER}) and since $Y\to S$ has separated diagonal by \cite[Tag 02X4]{stacks-project}, we see $\beta$ has affine diagonal. Then applying \cite[Proposition 3.13]{Alper}, we see $\alpha$ is cohomologically affine; since it is a gerbe, it is also Stein hence a good moduli space map.  Furthermore, the Keel--Mori Theorem tells us $\beta$ is proper. 
		Since $\beta$ has finite inertia, hence quasi-finite diagonal, by \cite[Theorem 2.7]{EHKV}, there is a finite cover $Z\to\cY$ by a scheme. Then $Z\to Y$ is proper quasi-finite, hence finite.
		
		Thus, replacing $Y$ by $Z$, we may assume $\pi$ is a gerbe.
	\end{proof}
	
	We now turn to our main theorems.
	
	\begin{proof}[{Proof of Theorems \ref{thm:main-gms-splitting} and \ref{thm:main-gms-splitting-nonNoeth}}]
		By Lemma \ref{l:SNoeth}, we have reduced Theorems \ref{thm:main-gms-splitting} and \ref{thm:main-gms-splitting-nonNoeth} to the special case of Theorem \ref{thm:main-gms-splitting} where $Z\to Y$ is the identity map, $Y$ is irreducible, and $S$ is finite type over a localization of $\mathbb{Z}$. Then Proposition \ref{prop:gerbe-case} allows us to assume $\pi$ is a gerbe. By Proposition \ref{prop:gerbe-connected} and Corollary \ref{cor:gerbe-conn-split}, we may assume $\pi$ has connected reductive stabilizers at all field valued points. Then \cite[Ch.~V, sect.~3.2, p.~75]{Douai} says 
		$\pi$ is necessarily banded by a reductive group scheme $G\to Y$. Applying Proposition \ref{prop:gerbe-abelian}, it suffices to prove Theorem \ref{thm:main-gms-splitting} for all $Z(G)$-gerbes. Another application of Proposition \ref{prop:gerbe-connected} allows us to replace $Z(G)$ with $Z(G)^0_{\sm}$, hence we may assume $\pi$ is a gerbe banded by a smooth connected multiplicative group scheme $G\to Y$. 
		
		Further replacing $Y$ by its normalization and considering each connected component individually, we can assume $Y$ is normal and irreducible. Then \cite[Corollary B.3.6]{ConradReductive} tells us $G\to Y$ is isotrivial, i.e., after replacing $Y$ by a finite \'etale cover, we may assume $G\to Y$ is isomorphic to $\bG_m^r$. Then $\pi$ is classified by an element of 
		\[
		H^2_{\textrm{\'et}}(Y,G)= H^2_{\textrm{\'et}}(Y,\bG_m)^{\oplus r}.
		\]
		It therefore suffices to show that any $\alpha \in H^2_{\textrm{\'et}}(Y,G)$ can be annihilated by an alteration of $Y$. Let $F = k(Y)$ be the function field and consider $\alpha|_F$. By \cite[Cor. 1.3.6]{ColliotTheleneSkorobogatov}, $\alpha|_F$ is torsion and there exists a finite extension $F'/F$ which annihilates it. Let $Y' \to Y$ be the normalization of $Y$ inside $F'$ and $Y'' \to Y'$ be an alteration of $Y'$ with $Y''$ regular which exists by Chow's lemma for algebraic spaces \cite[IV Theorem 3.1]{Knutson1971} and \cite[Theorem 1.2.5]{Temkin2017}. Then the pullback of $\alpha$ to the function field of $Y''$ vanishes by construction. Thus $\alpha|_{Y''} = 0$ by injectivity of $H^2_{\textrm{\'et}}(Y'', \mathbb{G}_m) \to H^2_{\textrm{\'et}}(k(Y''), \mathbb{G}_m)$ \cite[Theorem 3.5.4]{ColliotTheleneSkorobogatov}. 
		
	\end{proof}
	
	Note that Remark \ref{rmk:properlystable-factorization}(i) holds since we began Proposition \ref{prop:gerbe-case} by replacing $\cX$ with $\cX^{\ps}$.

	\section{Proof of Theorem \ref{thm:main-gms-splitting-rational-version}}
	
	The proof of Theorem \ref{thm:main-gms-splitting-rational-version} is similar to that of Theorem \ref{thm:main-gms-splitting}. To begin, the first and last paragraphs of the proof of Lemma \ref{l:SNoeth} shows we can again assume $Z\to Y$ is the identity map and $Y$ is irreducible and separated over $S$.

	We next follow the argument in Proposition \ref{prop:gerbe-case}, however the proof needs some modifications. First, the assumption that the generic point of $Y$ maps to $\cX^{\ps}$ allows us to replace $\cX$ by $\cX^{\ps}$. We may then apply \cite[Theorem 2.11]{ER} as before. Note that the maps $p$ and $p'$ do not change $\cX$ or $Y$ generically, so our generic section of $\pi$ defines a generic section of $\pi'\colon\cX'\to Y'$. The remainder of the proof of Proposition \ref{prop:gerbe-case} is the same with one notable change. Previously we applied \cite[Theorem 2.7]{EHKV} to find a splitting for the map $\beta$; we must now ensure we can choose this splitting in such a way that it extends the generic section. Thus, we replace the use of \cite[Theorem 2.7]{EHKV} with the following.
	
	\begin{proposition}\label{prop:tame-DM-proper-generic-splitting-extension}
		Suppose $\pi\colon\cX\to Y$ is the coarse space map of an irreducible finite type separated tame Artin stack over a Noetherian algebraic space $S$.  
		If $i\colon U\hookrightarrow Y$ is the inclusion of a dense open subspace and $f\colon U\to\cX$ a morphism with $\pi f=i$, then there exists a finite cover $h\colon Z\to Y$ and an extension of $f|_{Z_U}$ to a map $g\colon Z\to\cX$ with $\pi g=h$.
	\end{proposition}
	\begin{proof}
		First, $\pi$ is proper by the Keel--Mori theorem since $\cX$ has finite inertia. Then $f$ is separated since $\pi$ is proper and $i=\pi f$ is separated by \cite[Tag 01L7]{stacks-project}. We also see $f$ is quasi-finite since $\pi f=i$ is. Since $\cX$ is irreducible, $Y$ is as well; it follows that $i$ is finite type hence $f$ is too. Then $f$ is separated, finitely presented, and representable, so Zariski's Main Theorem \cite[Theorem 16.5]{LMB} tells us there is an open immersion $j\colon U\hookrightarrow\cZ$ and a finite map $p\colon\cZ\to\cX$.
		
		Next, \cite[Theorem 2.7]{EHKV} says there is a finite cover $Z\to\cZ$. Then $Z\to Y$ is proper quasi-finite, hence finite. Note that $U'\colon=U\times_\cZ Z\to Z$ is an open immersion with dense image and the map $Z\to\cX$ agrees with the map $U'\to U\to\cX$. Lastly, we must show that the map $Z\to\cX$ when restricted to the open subset $U''\colon=U\times_Y Z$ extends $f$. For this, it suffices to show the natural map $U'\to U''$ is open immersion. For this, note that $U'\to Z$ and $U''\to Z$ are \'etale monomorphisms, hence $U'\to U''$ is as well, and is therefore an open immersion by \cite[Tag 025G]{stacks-project}.
	\end{proof}
	
	As before, we next reduce to the case where $\pi$ is banded by a reductive group scheme $G\to Y$ by Proposition \ref{prop:gerbe-connected}, Corollary \ref{cor:gerbe-conn-split}, and \cite[Ch.~V, sect.~3.2, p.~75]{Douai}. Here the proof of Corollary \ref{cor:gerbe-conn-split} needs to be modified, replacing the use of \cite[Theorem 2.7]{EHKV} by Proposition \ref{prop:tame-DM-proper-generic-splitting-extension}. We then apply Proposition \ref{prop:gerbe-abelian}, Proposition \ref{prop:gerbe-connected}, and Corollary \ref{cor:gerbe-conn-split} to reduce to the case of tori. One need only check that in the setting of Proposition \ref{prop:gerbe-abelian}, our given morphism $U \to \cX$ lifts to a $U$-point of $\sI = \uIsom_Y^{\varphi}(\cX, BG)$ and a $U$-point of $\cX$ induces a banded isomorphism $\cX|_U \cong BG|_U$ that sends the given $U$-point to the canonical section of $BG|_U$ as required.

	The same argument reduces us to the case $G=\bG_m$. Since $S$ is finite type over a quasi-excellent scheme of dimension at most $3$, Chow's lemma for algebraic spaces \cite[IV Theorem 3.1]{Knutson1971} and \cite[Theorem 1.2.5]{Temkin2017} shows that after an alteration, we may assume $Y$ is regular. Then as before, by purity for Brauer groups \cite[
	Theorem 3.5.4]{ColliotTheleneSkorobogatov}, we have injectivity of the map
	\[
	H^2_{\textrm{\'et}}(Y,\bG_m)\hookrightarrow H^2_{\textrm{\'et}}(U,\bG_m),
	\]
	and thus our generically trivial $\mathbb{G}_m$-gerbe is trivial over $Y$ and the given section over $U$ which we have identified with the canonical section of $B\mathbb{G}_m$ extends to $Y$. 
	
	\begin{proof}[Proof of Corollary \ref{cor:compactification}]
		
		By Chow's lemma, we can assume $U$ is quasi-projective. Let $W$ be some projective compactification and let $Z$ be the closure of the graph of $U \to Y$ inside $W \times_S Y$. Then $Z$ is a projective compactification of $U$ with a morphism $Z \to Y$ and lift $s_U : U \to \cX$. By the above argument, there exists an alteration $Z' \to Z$ with $U' \colon = U \times_Z Z'$ and an extension of the induced morphism $U' \to \cX$ to a morphism $Z' \to \cX$ as required. 
	\end{proof}
	
	\begin{proof}[Proof of Corollary \ref{cor:variation}]
		
		Let $Y' \subset Y$ be the scheme theoretic image of $Z \to Y$. Up to replacing $\cX$ with $\cX' = Y' \times_{Y}\cX$, we may assume that $Z \to Y$ is dominant, $Y$ is irreducible and $\Var g = \dim Y$. By Theorem \ref{thm:main-gms-splitting}, there exists an alteration $Y' \to Y$ with $Y'$ irreducible a map $f' \colon= Y' \to \cX$. Let $Z'$ be a component of the fiber product $Z \times_Y Y'$ which dominates both $Z$ and $Y'$. Then taking $Z'' \colon = Y'$, $Z' \to Z$, $Z' \to Z''$, and $g'' = f'$ are as required. 
		
	\end{proof}
	
	\section{Projectivity of good moduli spaces:~proof of Theorem \ref{thm:kollar-ampleness}}
	
	We next apply Theorem \ref{thm:main-gms-splitting} to generalize \cite[Theorem 4.5]{KovacsPatakfalvi2017} and Koll\'ar's ampleness lemma \cite[Lemma 3.9]{KollarAmpleness}, obtaining a criterion for projectivity of good moduli spaces.
	
	\begin{proof}[{Proof of Theorem \ref{thm:kollar-ampleness}}]
		We apply Nakai--Moishezon. Let $i\colon Z\hookrightarrow Y$ be a $d$-dimensional integral closed subalgebraic space of $Y$. We must show the intersection product $(i^*\cL)^d>0$. By Theorem \ref{thm:main-gms-splitting}, there exists an alteration $f\colon W\to Z$ and a map $s\colon W\to\cX$ such that $\pi s=if$. We may assume $W$ is normal after replacing it with its normalization and projective by applying Chow's lemma. Furthermore, by Remark \ref{rmk:properlystable-factorization}, we may choose $s$ so that the generic point of $W$ factors through $\cX^{\ps}$. By the projection formula, it suffices to prove $(f^*i^*\cL)^d>0$. 
		
		We see that $f^*i^*\cL=\det(s^*\cQ)^{\otimes N}$ and that $s^*\alpha\colon s^*\cE\twoheadrightarrow s^*\cQ$ is a quotient bundle. The induced classifying map $[s^*\alpha] \colon W \to [\Gr(n,q)/G]$ factors as $[\alpha] \circ s$. We claim that $[s^*\alpha]$ is set theoretically generically finite. It suffices to check this on algebraically closed field $\overline{k}$-points. Then $Z(\overline{k}) \hookrightarrow |\cX^{ps}(\overline{k})|$, $W(\overline{k}) \to Z(\overline{k})$ is generically finite, and $[\alpha](\overline{k})$ is set theoretically finite on polystable points by assumption so the composition $[\alpha] \circ s$ is set theoretically generically finite. 
		
		Next we check that $s^*\cE$ is weakly-positive in the sense of Viehweg (\cite[Definition 3.7]{KovacsPatakfalvi2017}). By \cite[\S 3.4(d)]{viehweg}, $s^*\cE$ is weakly-positive if and only if $\cO_{\bP(s^*\cE)}(1)$ is weakly-positive on $\bP(s^*\cE)$. For a line bundle $\cN$, weakly positive is equivalent to $N := c_1(\cN)$ pseudoeffective. Indeed $\cN$ is weakly-positive if and only if $\cN^{\alpha \beta}\otimes \cH^\beta$ is effective for any $\alpha > 0$ and any $\beta$ large enough (depending on $\alpha$) where $\cH$ is an ample line bundle. Writing additively and dividing by $\beta$, we see that this is the case if and only if $N + \frac{1}{\alpha}H$ is an effective $\mathbb{Q}$-divisor for any $H = c_1(\cH)$ ample and any $\alpha > 0$. This holds if and only if $N$ is pseudoeffective by \cite[Theorem 2.2.26]{positivity}. 
		
		Since the generic point of $W$ maps to the polystable which is constructible, there exists a dense open set $U \subset W$ such that $s|_U \colon U \to \cX$ factors through $\cX^{\ps}$. By condition (3), for any curve $C \subset W$ which meets $U$, any line bundle quotient of $s^*\cE|_C$ has non-negative degree. Thus $N:=c_1(\cO_{\bP(s^*\cE)})$ has non-negative degree on any lift of any curve $C$ meeting $U$. In particular, $N$ has non-negative degree on any movable curve on $\bP(s^*\cE)$ so by \cite[Theorem 0.2]{bdpp} and its extension to positive characteristic \cite[Theorem 2.22]{fulgerlehmann}, $N$ is pseudoeffective so $s^*\cE$ is weakly-positive. 
		
		Then by Remark 4.4 and Theorem 4.5 of \cite{KovacsPatakfalvi2017}\footnote{Note that the only place characteristic $0$ is used in \emph{loc. cit.} is in \cite[Proposition 4.7]{KovacsPatakfalvi2017} and this is taken care of in positive characteristic by the conditions in assumption (3) (see \cite[Remark 4.8]{KovacsPatakfalvi2017}).} (which is a more general version of \cite[Lemma 3.9]{KollarAmpleness} where the finiteness assumption on the classifying map is weakened), $\det(s^*\cQ)$ is big and nef, hence $(f^*i^*\cL)^d>0$ as required.
	\end{proof}
	
	\begin{remark}
		There are some other versions of the ampleness lemma in the literature that are applied to the good moduli space of an Artin stack in specific examples. They are used to prove ampleness of the CM line bundle (and thus projectivity) of the good moduli space of the stack $\cM^{Kss}$ of K-semistable Fano varieties \cite{codognipatakfalvi, posva, xuzhuang}. Splittings of good moduli spaces after an alteration (Theorem \ref{thm:main-gms-splitting}) can be used to extend and streamline some of the arguments there (\cite[Section 10]{codognipatakfalvi} and \cite[Section 7.1]{xuzhuang}). 
	\end{remark}

	\bibliographystyle{alpha}
	\bibliography{gms-base-change}
	
\end{document}